\documentclass{article}
\pdfoutput=1

\usepackage{a4wide,amsthm, color, verbatim, multicol, hanshopf, float}
\usepackage{hyperref}
\usepackage{difftrees}
 \usepackage{amsmath}
 \usepackage{latexsym}
 \usepackage{amssymb}
 \usepackage{amscd}
 \usepackage{stmaryrd}
 \usepackage{shuffle}
 \usepackage{wasysym}
 \font\gothic=eufm10 scaled 1100

\usepackage[small]{diagrams}

\newtheorem{theorem}{\bf{Theorem}}[section]

\newtheorem{corollary}[theorem]{\bf{Corollary}}

\newtheorem{proposition}[theorem]{\bf{Proposition}}

\theoremstyle{definition}
\newtheorem{example}[theorem]{\bf{Example}}

\newtheoremstyle{remark}{\topsep}{\topsep}
{}
{}
{\bfseries}
{.}
{ }
{\thmname{#1}\thmnumber{ #2}\thmnote{ #3}}

\theoremstyle{definition}
\newtheorem{remark}[theorem]{\bf{Remark}}

\theoremstyle{definition}
\newtheorem{definition}[theorem]{\bf{Definition}}

\def \Hom{\operatorname{Hom}}

\def \Id{\operatorname{Id}}
\def \k{\operatorname{k}}

\def\csum{\ensuremath{\mbox{{\gothic S}}}}
\def\tr{\rhd}
\def\vartr{\RHD} 
\def\ctr{\Yright} 
\def\Bplus{\operatorname{B}^+}
\def\Bminus{\operatorname{B}^-}
\def\D{\operatorname{D}}

\def\ass{\operatorname{ass}}

\begin{document}


\thispagestyle{empty}

\title{On post-Lie algebras, Lie--Butcher series and moving frames}

\date{}

\author{Hans Z. Munthe-Kaas\footnote{\small Department of Mathematics, University of Bergen, Norway. {\small{Email: hans.munthe-kaas@math.uib.no}}} \and Alexander Lundervold\footnote{\small Inria Bordeaux Sud-Ouest, France. {\small{Email: alexander.lundervold@gmail.com}}}}




\maketitle



\begin{abstract} 
\emph{Pre-Lie}  (or Vinberg) algebras arise from flat and torsion-free connections on differential manifolds. They have been extensively studied in recent years, both from algebraic operadic points of view and through numerous applications in numerical analysis, control theory, stochastic differential equations and renormalization. Butcher series are formal power series founded on pre-Lie algebras, used in numerical analysis to study geometric properties of flows on euclidean spaces. Motivated by the analysis of flows on manifolds and homogeneous spaces, we investigate algebras arising from flat connections with constant torsion, leading to the definition of \emph{post-Lie} algebras, a generalization of pre-Lie algebras. Whereas pre-Lie algebras are intimately associated with euclidean geometry, post-Lie algebras occur naturally in the differential geometry of homogeneous spaces, and are also closely related to Cartan's method of moving frames. Lie--Butcher series combine Butcher series with Lie series and are used to analyze flows on manifolds. In this paper we show that Lie--Butcher series are founded on post-Lie algebras. The functorial relations between post-Lie algebras and their enveloping algebras, called D-algebras, are explored. Furthermore, we develop new formulas for computations in free post-Lie algebras and D-algebras, based on recursions in a magma, and we show that Lie--Butcher series are related to invariants of curves described by moving frames.
\end{abstract}




\section{Introduction} 
In 1857 Arthur Cayley published a remarkable paper~\cite{cayley1857taf} introducing the idea of using trees to describe differential operators. This became a fundamental tool in the analysis of numerical flows through the seminal work of John Butcher~\cite{butcher1963cft,butcher1972aat} in the 1960s, where he introduced the Butcher group, or, in modern language, the character group of the Butcher--Connes--Kreimer Hopf algebra. In 1963 \emph{pre-Lie algebras} appeared simultaneously from two different sources, Vinberg \cite{vinberg1963chc} from differential geometry (classification of homogeneous cones) and Gerstenhaber \cite{gerstenhaber1963tcs} from algebra (Hochschild cohomology). Pre-Lie algebras later appeared in control theory under the name \emph{chronological algebras} \cite{agrachev1981chronological}. 

B-series, named after John Butcher in~\cite{hairer1974butcher}, are formal infinite series used to study properties of numerical integration algorithms on a euclidean space $V$. Over the last decades, B-series have evolved into powerful algebraic tools used for studying a variety of geometric properties of flows, such as symplecticity, volume and energy preservation, preservation of first integrals and backward error analysis. In the late 1990s similar structures appeared in the combinatorial approach to renormalisation by  Connes and Kreimer~\cite{connes1998har}.  In \cite{brouder2000rkm}, Christian Brouder pointed out connections between the theory of Connes--Kreimer and numerical analysis. 

A B-series is indexed over the set of rooted trees (non-planar trees, where the order of the branches is neglected),
${\mathcal T} = \left\{\ab, \aabb,\aaabbb, \aababb, \aaaabbbb,\aaababbb,\aaabbabb, \aabababb,\ldots
\right\}$.
For a given vector field $f\colon V\rightarrow V$ and a real valued function $\alpha\colon {\mathcal T}\rightarrow \RR$, a B-series is defined as
\begin{equation}\label{eq:bser}{\mathcal B}_f(\alpha) =  \sum_{\tau\in {\mathcal T}}\alpha(\tau)\F(\tau), \end{equation}
where
 $\F(t)$, called the \emph{elementary differentials},  are vector fields on $V$ defined recursively for all  $\tau\in\mathcal T$, starting with $\F(\ab) = f$. The recursive extension of $\F$ to all of $\mathcal T$ can be obtained from the theory of pre-Lie algebras.
 
Pre-Lie algebras provide an axiomatic foundation of B-series. A pre-Lie algebra is an abstract algebra $\{A,\tr\}$
with a product $\tr$ satisfying a skew associator relationship given in~(\ref{eq:prelie}). 
The free pre-Lie algebra~\cite{chapoton2001pla} is the algebra spanned by all rooted trees, where the product $\tr$
is given as tree-grafting: attach the root of the left tree in all possible ways to the nodes of the right tree and sum everything. An example is
\begin{equation*}
  \aabb\tr\aababb=\aaabbababb +\aaaabbbabb+\aabaaabbbb=\aaabbababb +2\aaaabbbabb.
\end{equation*}
As we will detail in Section~\ref{subsec:conn}, the set of vector fields on a euclidean space also forms a pre-Lie algebra, with the product $x\tr y:= \nabla_x y$, where $\nabla$ is a flat and torsion free connection associated with the euclidean structure.
In this light, a B-series is an instance of a formal infinite series in a pre-Lie algebra, and the elementary differential map $\F$ is the universal morphism from the free pre-Lie algebra to the pre-Lie algebra of euclidean vector spaces, defined by the recursion $\F(\ab) = f$, $\F(\tau\tr \tau') = {\F(\tau)}\tr\F(\tau')$.

During the 1990s numerical integration was generalized from euclidean spaces to manifolds~\cite{crouch1993nio,munthe-kaas1995lbt,munthe-kaas1998rkm,owren1999rkm}. In this work it was necessary to introduce a generalization of B-series to manifolds, called Lie--Butcher series (LB-series). Inspired by the unexpected connections between 
numerical analysis and renormalization, the algebraic foundations of LB-series were investigated in several papers in the last 
decade~\cite{berland2005aso,lundervold12oas,lundervold2009hao,munthe-kaas2003oep,munthe-kaas2008oth}. A Hopf algebra unifying the Butcher--Connes--Kreimer and the shuffle Hopf algebra of a free Lie algebra emerged. This Hopf algebra is defined on the linear space spanned by \emph{forests} of planar trees. The underlying algebraic structure was termed a D-algebra, which is a generalization of the \emph{enveloping algebras} of  pre-Lie algebras. However, the definition of this underlying algebraic structure and the exact understanding of D-algebras as an universal enveloping algebra of  `something', has so-far been an open problem. The present paper rectifies this.

The study of `something' leads to the definition of \emph{post-Lie} algebras, first found by Bruno Vallette~\cite{vallette2007hog} in 2007 through the purely operadic technique of Koszul dualization. In this paper we show that post-Lie algebras also arise naturally from the differential geometry of homogeneous spaces and Klein geometries, topics that are closely related to Cartan's method of \emph{moving frames}. 
Applications of moving frames in computational mathematics have been pioneered by Peter Olver and his co-workers~\cite{fels1998moving,fels1999moving,olver1995equivalence}. 
We show that post-Lie algebras and LB-series are related to moving frame theory, and point to possible applications of moving frames in the design of numerical methods.
Furthermore, we pursue a detailed description of post-Lie algebras, their universal enveloping algebras and the free post-Lie algebra. Finally, we develop some new interpretations of LB-series in the light of moving frames. This leads to a lot of questions about the design of numerical algorithms, of both a theoretical and a practical nature. These are not discussed in this paper.

The paper is organized as follows. Section 2 motivates the definition of post-Lie algebras and presents post-Lie algebras from differential geometric points of view. Several concrete examples are examined, such as post-Lie algebras associated with Cartan-Schouten connections on Lie groups, post-Lie structures on vector bundles over homogeneous manifolds and relations to moving frames. 
In Section~\ref{sec:postliealgebroid} we define a structure we call \emph{post-Lie algebroid}, which generalizes all the specific examples from differential geometry.
Section 3 is devoted to algebraic properties of post-Lie algebras and related structures. Sections 2 and 3 can be read  independently of each other. Section 4 gives a brief presentation of LB-series, where the topics of Sections 2 and 3 are combined in a discussion of different ways of representing flows on manifolds in terms of their LB-series, and some final remarks on the relationship between LB-series, developments of curves and moving frames.

\section{Post-Lie algebras in differential geometry}
\subsection{Algebras of  connections}\label{subsec:conn}
In this section we motivate the abstract definition of pre-Lie, post-Lie and Lie-admissible algebras by considering algebras of vector fields originating from  connections on a manifold.
In the sequel, we assume that all algebras are vector spaces over a field $\k$ of characteristic 0, e.g.\ $\k\in \{\RR,\CC\}$. 

The most fundamental concept in differential geometry is \emph{connections}, defining parallel transport and covariant derivations. Connections appear in various abstractions, e.g.\ Koszul, Ehresmann and Cartan connections. To motivate pre-Lie, post-Lie and Lie-admissible algebras, it is sufficient to consider the simplest definition: a  Koszul connection on the tangent bundle. 

Let $\XM$ denote the smooth vector fields on a manifold $\M$. A \emph{Koszul connection}~\cite{spivak2005aci}
 is defined as a map $\nabla\colon \XM\times \XM \rightarrow\XM$ such that
 \begin{align*}\nabla_{fx} y &= f \nabla_x y \\
\nabla_x (fy) &= df(x) y + f\nabla_x y,
\end{align*}
for any $x,y\in \XM$ and scalar field $f$.
The connection defines a (non-commutative and non-associative) $\RR$-bilinear product on $\XM$, henceforth written as \[x\tr  y := \nabla_x y.\]
The \emph{torsion} of the connection is a skew-symmetric tensor $T\colon \TM\wedge \TM\rightarrow \TM$ defined as
\begin{equation}T(x,y) = x\tr y -y\tr x - \llbracket x,y\rrbracket_J ,\label{eq:torsion}\end{equation}
 where $\llbracket\cdot,\cdot\rrbracket_J$ denotes the Jacobi--Lie bracket of vector fields, defined such that  $\llbracket x,y\rrbracket_J(\phi) = x(y(\phi))-y(x(\phi))$ for all vector fields $x,y$ and scalar fields $\phi$.
The \emph{curvature tensor} $R\colon \TM\wedge \TM \rightarrow \mbox{End}(\TM)$ is defined as
\begin{equation}R(x,y)z = x\tr (y\tr z)- y\tr (x\tr z)- {\llbracket x,y\rrbracket_J}\tr z  = a_\tr(x,y,z)-a_\tr(y,x,z)+T(x,y)\tr z,\label{eq:curvature}
\end{equation}
where $a_\tr(x,y,z)$ is the (negative) \emph{associator}  of the product $\tr$, defined for any product $\ast$ as
\begin{equation} a_{\ast}(x,y,z) := x\ast (y\ast z)-(x\ast y)\ast z.\label{eq:associator}\end{equation}
The torsion (resp.\ curvature) is \emph{constant} if the covariant derivative vanishes, $\nabla T=0$ (resp.\ $\nabla R=0$).
\noindent The relationship between torsion and curvature is given by the Bianchi identities
\begin{eqnarray}
\csum(T(T(x,y),z)+(\nabla_xT)(y,z)) &=& \csum(R(x,y)z)\label{eq:bianchi1}\\
\csum((\nabla_xR)(y,z)+R(T(x,y),z)) & =& 0 ,\label{eq:bianchi2}
\end{eqnarray}
where $\csum$ denotes the sum over the three cyclic permutations of $(x,y,z)$.

\paragraph{Torsion free connection  $\Rightarrow$ Lie-admissible algebra.} If $T=0$ and $\nabla R=0$ then $R(x,y)z = a_\tr(x,y,z)-a_\tr(y,x,z)$, and the Bianchi identities reduce to a single equation
\begin{equation}\csum(a_\tr(x,y,z)-a_\tr(y,x,z))=0 .\label{eq:lie-adm}\end{equation} 
A general algebra $\{A,\ast\}$ with a product $\ast$ satisfying~(\ref{eq:lie-adm}) is called a \emph{Lie-admissible algebra}. Lie-admissible algebras are exactly those algebras that give rise to Lie algebras by skew-symmetrization of the product~\cite{goze2004laa}. In this example skew symmetrization yields the Jacobi--Lie bracket
\[ x\tr y-y\tr x = \llbracket x,y\rrbracket_J .\]
Any associative algebra is clearly Lie-admissible. A more general example
is pre-Lie algebras,  defined below.

\paragraph{Flat and torsion free connection $\Rightarrow$ Pre-Lie algebra.} Consider a connection which is both flat $R=0$ and torsion free $T=0$.  Equation~(\ref{eq:curvature}) implies that
\begin{equation}a_\tr(x,y,z)-a_\tr(y,x,z)=0 .\label{eq:prelie}
\end{equation}
A general algebra $\{A,\tr\}$ with a product $\tr$  satisfying~(\ref{eq:prelie}) is called a \emph{pre-Lie algebra} or a \emph{Vinberg algebra}.
Pre-Lie algebras appear in many settings. A fundamental algebraic result is that the free pre-Lie algebra can be described as the set of rooted trees with product given by grafting~\cite{chapoton2001pla}.

Note that a Riemannian manifold with a flat, torsion free connection is locally isometric to a euclidean space $\RR^n$ with the standard metric~\cite{spivak2005aci}, hence pre-Lie algebras are tightly associated with the differential geometry of euclidean spaces. For Lie groups and homogeneous spaces (Klein geometries), pre-Lie algebras are not sufficiently general to capture the basic differential geometry algebraically.

\paragraph{Flat and constant torsion connection $\Rightarrow$ Post-Lie algebra.}
Given a connection which is flat $R=0$ and has constant torsion $\nabla T=0$, then~(\ref{eq:bianchi1}) reduces to the Jacobi identity $\csum(T(T(x,y),z)) = 0$. Recall that a \emph{Lie bracket} is a skew symmetric bilinear form satisfying the Jacobi identity,  hence the torsion defines a Lie bracket (see Remark~\ref{rem:sign} about the negative sign)
\begin{equation}[x,y] := -T(x,y).\end{equation}
Note that this \emph{torsion bracket} is related to the Jacobi-Lie bracket by~(\ref{eq:torsion}).

The covariant derivation formula $
\nabla_x (T(y,z)) = (\nabla_x T)(y,z)+T(\nabla_x y,z)+T(y,\nabla_x z)$ together with $\nabla_x T=0$ imply
\begin{equation}x\tr [y,z] = [x\tr y, z]+[y,x\tr z] .\label{eq:postlie1}\end{equation}
On the other hand, (\ref{eq:curvature}) together with $R=0$ imply
\begin{equation}[x,y]\tr z = a_\tr(x,y,z)-a_\tr(y,x,z) .\label{eq:postlie2}\end{equation}
This motivates the general abstract definition of a \emph{post-Lie algebra}.
\begin{definition}A post-Lie algebra $\{\A,[\cdot,\cdot],\tr\}$ is a Lie algebra $\{\A,[\cdot,\cdot]\}$ together with a product $\tr\colon \A\times\A\rightarrow \A$ 
such that~(\ref{eq:postlie1})-(\ref{eq:postlie2}) hold. We call $[\cdot,\cdot]$ the \emph{torsion} and $\tr$ the \emph{connection} of the post-Lie algebra.
\end{definition}

\noindent{\bf Note:} Henceforth $\tr$ and $[\cdot,\cdot]$ denote these abstractly defined torsion and connection on a post-Lie algebra, of which the  examples of torsion and connections on a manifolds are concrete special cases.

\begin{remark}\label{rem:sign}In many applications one may naturally obtain~(\ref{eq:postlie2}) with opposite sign
\[[x,y]\tr z = a_\tr(y,x,z)-a_\tr(x,y,z).\]
We could have defined \emph{left} and \emph{right} post-Lie algebras according to these sign changes. This terminology  makes sense e.g.\ in the case of the Maurer--Cartan form on a Lie group.
However, the sign in~(\ref{eq:postlie2}) can always be absorbed into a change of the sign in the definition of the torsion bracket, since $\{\A,-[\cdot,\cdot]\}$ is also a Lie algebra and~(\ref{eq:postlie1}) is invariant under a sign change in the torsion. There is therefore no need to define both left- and right-versions of post-Lie algebras.
\end{remark}

\begin{remark} Post-Lie algebras were  introduced around 2007 by  B.\ Vallette~\cite{vallette2007hog}, who found the structure in a purely operadic manner as the Koszul dual of the operad governing commutative trialgebras. The enveloping algebra of a post-Lie algebra was independently introduced about the same time under the name $D$-algebra in~\cite{munthe-kaas2008oth}, and studied in the context of Lie--Butcher series. A basis for the free post-Lie algebra was presented in~\cite{munthe-kaas2003oep}, before the post-Lie definition was formalized. 
Vallette defines a post-Lie operad  and proves that post-Lie algebras have the important algebraic property of being Koszul. This property is shared by many other fundamental algebras, such as Lie algebras, associative algebras, commutative algebras, pre-Lie algebras, dendriform algebras etc. He also defines the operadic homology of post-Lie algebras and computes this for the free post-Lie algebra.
\end{remark}

\begin{remark} 
The present differential geometric explanation of a post-Lie algebra is, as far as we are aware, new.  Since
condition~(\ref{eq:postlie2}) is expressing the flatness of the connection, while (\ref{eq:postlie1}) derives from the constant torsion, we initially called this structure a FCT (flat, constant torsion) algebra, but will henceforth adhere to the name \emph{post-Lie algebra}.

A pre-Lie algebra is a post-Lie algebra where $[\cdot,\cdot]=0$, hence most results about post-Lie algebras also yield information about pre-Lie algebras. In particular, the
$D$-algebra provides a definition for the enveloping algebra of a pre-Lie algebra.  We return to this in the sequel.
\end{remark}

\subsubsection{Some basic results about post-Lie algebras.}
\begin{proposition}\label{prop:jlbracket}If $\{A,[\cdot,\cdot],\tr\}$ is post-Lie, then the bracket $\llbracket x,y\rrbracket$ defined as
 \begin{equation}\llbracket x,y\rrbracket := x\tr y - y\tr x + [x,y]\label{eq:jlb} \end{equation}
 is a Lie bracket.
\end{proposition}

\begin{proof}Clearly $\llbracket x,y\rrbracket=-\llbracket y,x\rrbracket$, and the Jacobi rule $\csum\left( \llbracket\llbracket x,y\rrbracket,z\rrbracket\right)=0$ is verified by a direct computation. \end{proof}
Note that in the motivating example of connections on manifolds, (\ref{eq:jlb}) is identical to (\ref{eq:torsion}), where $\llbracket\cdot,\cdot\rrbracket$ is the Jacobi-Lie bracket and $[\cdot,\cdot]$ the negative torsion.

By a modification of the product $\tr$ in $A$, we obtain another post-Lie algebra. 
\begin{proposition}\label{prop:twist}Let $\{A,[\cdot,\cdot],\tr\}$ be post-Lie.  Define the product $\vartr$ as
\[x\vartr y = x\tr y + [x,y],\]
then $\{A,-[\cdot,\cdot],\vartr\}$ is also post-Lie.
\end{proposition}
\begin{proof} Since both $x\tr\cdot$ and $[x,\cdot]$ are derivations on the torsion bracket, $x\tr \cdot + \alpha [x,\cdot]$ is also a derivation, for
any $\alpha\in \k$. A direct computation shows that~(\ref{eq:postlie2}) holds with a sign change, which is corrected by negating the torsion bracket. 
\end{proof}

\begin{proposition}\label{prop:levicivita}Let $\{A,[\cdot,\cdot],\tr\}$ be post-Lie.  Define the product $\ctr$ as
\[x\ctr y = x\tr y + \frac12[x,y],\]
then $\{A,\ctr\}$ is Lie-admissible. The skew associator is
\[a_\ctr(x,y,z)-a_\ctr(y,x,z)= -\frac14[[x,y],z] =: R(x,y)z,\]
which can be interpreted as a curvature. 
\end{proposition}

\begin{proof}Lie-admissibility means that the skew symmetrisation of the product is a Lie bracket. This follows from $x\ctr y - y\ctr x = \llbracket x,y\rrbracket$, defined in~(\ref{eq:jlb}). The formula for the skew associator follows from
\begin{align*}
R(x,y)z &:= x\ctr(y\ctr z) - x\leftrightarrow y - {\llbracket x,y\rrbracket}\ctr z \\ &= x\tr (y\tr z + \frac12 [y,z]) + \frac12 [x,y\tr z + \frac12 [y,z]]
- x\leftrightarrow y - \llbracket x,y\rrbracket \tr z - \frac12 [\llbracket x,y\rrbracket,z] \\
&= \frac12 [x\tr y,z] + \frac12 [y,x\tr z] + \frac12 [x,y\tr z]+ \frac14[x,[y,z]] - x\leftrightarrow y-\frac12[\llbracket x,y\rrbracket,z] \\
&= \frac14[x,[y,z]]-x\leftrightarrow y -\frac12[[x,y],z] = \frac14 [[x,y],z] - \frac12 [[x,y],z] \\
&= -\frac14 [[x,y],z] ,
\end{align*}
where $x\leftrightarrow y$ means swap $x$ and $y$ in everything to the left. Note that the algebraic condition of Lie-admissibility, $\csum\left(a_\ctr(x,y,z)-a_\ctr(y,x,z)\right)=0$, follows from the Jacobi identity of $[\cdot,\cdot]$.

If $\llbracket\cdot,\cdot\rrbracket$ is interpreted as a Jacobi-Lie bracket of vector fields, it follows from (\ref{eq:torsion})-(\ref{eq:curvature}) that $T=0$ and $R(x,y)z$ is the curvature tensor.
\end{proof}

\begin{remark}\label{rem:cs}In the case of vector fields on a Lie group,  $\ctr$ corresponds to the (torsion free) Levi--Civita connection, while  $\tr$ and $\vartr$ correspond to connections obtained by trivializing the tangent bundle $TG$ from right or left (using the right and left Maurer--Cartan forms). This is detailed in Section~\ref{subsec:mcform}. These three connections on Lie groups were first introduced by Cartan and Schouten~\cite{cartan1926geometry}, and are often called the (+), (-) and (0) Cartan--Schouten connections. The differential geometric computation of the torsion and curvature for these is found in~\cite[Proposition~2.12]{kobayashi1969foundations}. It is remarkable that these objects can be defined in the abstract formalism of post-Lie algebras, and computed in this purely algebraic setting.
\end{remark}

\subsection{Lie groups and homogeneous spaces}\label{LieGps} 

\subsubsection{Post-Lie algebras and numerical Lie group integrators}\label{sec:lgi} We will briefly review the fundamentals  of numerical Lie group integration on a homogeneous space, as formulated in~\cite{munthe-kaas1999hor}, and reveal the underlying post-Lie algebra of this formulation.
Let $G$ be a Lie group with Lie algebra $\g$, and let  
\[\lambda\colon G\times \M\rightarrow \M,\quad (g,u)\mapsto g\cdot u\] be a transitive left action of $G$ on a homogeneous space $\M$, with infinitesimal generator
\begin{equation}\lambda_\star\colon \g\times \M\rightarrow T\M,\quad (v,u)\mapsto \left.\frac{\partial}{\partial t}\right|_{t=0}\exp(tv)\cdot u \in T_u\M.\label{eq:lambdastar}\end{equation}
Let $\Omega^k(\M,\g)$ be the space of $\g$-valued k-forms on $\M$. In particular, $\Omega^0(\M,\g)$ is the space of maps from $\M$ to $\g$. Any $x\in \Omega^0(\M,\g)$ generates a vector field $X\in \XM$ as
\begin{equation}X(u) = \lambda_\star(x(u),u) ,\label{eq:leftiso}\end{equation}
which by abuse of notation is written compactly as $X=\lambda_\star(x)$, where $\lambda_*\colon \Omega^0(\M,\g)\rightarrow\XM$. 
%

\begin{remark}
Most Lie group integrators for the differential equation $u'(t) = F(u(t))$, where $u(t)\in \M$ and $F\in \XM$, are based on rewriting the equation as $u'(t) = \lambda_* f(u(t))$ for $f\in \Omega^0(\M,\g)$. It is important to note that if the action of $G$ is transitive but not free on $\M$, then $\lambda_*\colon \Omega^0(\M,\g)\rightarrow\XM$ is surjective but not injective.
The freedom in choice of a $f$ to represent $F$ is given by the isotropy (stabilizer) subgroup of $G$ at a point $u\in \M$. Different choices of isotropy can lead to significantly different numerical integrators. As pointed out by Lewis and Olver~\cite{lewis2002gia}, moving frames is an important tool in the study of isotropy choice for Lie group integrators. We return to this in the sequel.
\end{remark}

Numerical analysis of Lie group integrators is intimately related to post-Lie algebras, due to the following result. 

\begin{proposition}\label{prop:homogen} Let $G$ be a Lie group, with Lie algebra $\{\g,[\cdot,\cdot]_\g\}$, acting from the left on a manifold $\M$. Let the Lie bracket 
$[\cdot,\cdot]\colon \Omega^0(\M,\g)\times \Omega^0(\M,\g)\rightarrow \Omega^0(\M,\g)$ and the product $\tr\colon \Omega^0(\M,\g)\times \Omega^0(\M,\g)\rightarrow \Omega^0(\M,\g)$ be defined pointwise at $u\in \M$ as
\begin{eqnarray}
[x,y](u) &=& [x(u),y(u)]_\g\\
x\tr y & = & \lambda_*(x)(y)\quad\mbox{(the Lie derivative of $y$ along $\lambda_*(x)$).}\label{eq:triangle}
\end{eqnarray}
Then $\{\Omega^0(\M,\g),[\cdot,\cdot],\tr\}$ is a post-Lie algebra.
\end{proposition}

\begin{proof} This can be verified by a coordinate computation. Let $\{e_j\}$ be a basis for $\g$ and $\partial_j=\lambda_*(e_j)$ the corresponding right invariant vector fields on $\M$. Note that $\lambda_*([e_j,e_k]_\g) = -\llbracket \partial_j,\partial_k\rrbracket_J$, where the right hand side is the Jacobi--Lie bracket of vector fields. Letting $x(p) = \sum_j x^j(p) e_j$ and $y(p)= \sum_k y^k(p) e_k$, where $x^j$ and $y^k$ are scalar functions on $\M$,  we obtain
\begin{eqnarray*}
[x,y] & = & \sum_{j,k}x^jy^k[e_j,e_k]\\
x\tr y & = & \sum_{j,k} x^j\partial_j(y^k)e_k.
\end{eqnarray*}
The post-Lie conditions follow by a straightforward computation, see~\cite[Lemma 3]{munthe-kaas2008oth} for a similar computation in the enveloping algebra. 
\end{proof}

The connection $\tr$ leads to a parallel transport of vector fields $\phi^*x (p) = x(\phi(p))$ for $x\in \Omega^0(\M,\g)$  and
$\phi\in \mbox{Diff}(\M)$. This parallel transport is used in numerical Lie group integrators to collect tangent vectors at a common
base point in order to compute the timestep of the algorithm. The parallel transport is  independent of paths, since the connection is flat, and it is given  as the exponential of $\tr$, which  can be expressed algebraically in the enveloping algebra of the post-Lie algebra of Proposition~\ref{prop:homogen}.

\subsubsection{The Maurer--Cartan forms}\label{subsec:mcform} 
In this section we return to the post-Lie structure obtained from the Maurer--Cartan forms on a Lie group, as mentioned in Remark~\ref{rem:cs}.
The (left) Maurer--Cartan (MC) form on a Lie group $G$ is
a $\g$-valued one-form  $\omega\in \Omega^1(G,\g)$, defined as the map $\omega\colon TG\rightarrow \g$ moving $v\in T_gG$ to $\g$ by left translation \[\omega(v) = TL_{g^{-1}} v ,\]
where $L_g h = g h$ is left multiplication in the group\footnote{The MC form can also be defined by right translation, but the left form is more convenient for moving frames.}. 

The Maurer--Cartan form defines a linear isomorphism $\omega_p\colon T_pG\rightarrow \g$ and hence defines an isomorphism between $\Omega^0(G,\g)$ and
vector fields  $\XG$. For the \emph{right} Maurer--Cartan form, this isomorphism is given by $\lambda_*$ defined in~(\ref{eq:lambdastar})-(\ref{eq:leftiso}) with $\M=G$. For the left MC form, the corresponding isomorphism   $\rho_*\colon\Omega^0(G,\g)\rightarrow \XG$ is given as the infinitesimal right action
\begin{equation}\label{eq:rhomap1}\rho_*(x)(g) = \left.\frac{\partial}{\partial t}\right|_{t=0}g \exp(t x(g)).\end{equation}

The  Maurer--Cartan form satisfies the structural equation
\begin{equation}d\omega + \frac12\omega\wedge\omega = 0.\label{eq:mcstructure}\end{equation}
On a general (connected, smooth) manifold $\M$, the
existence of a $\g$-valued one-form which is an isomorphism on the fibre and satisfies~(\ref{eq:mcstructure}) implies that $\M$ can be given the structure of a Lie group (up to a covering)~\cite[Theorem \S 8.8.7]{sharpe1997dg}. Thus the Maurer--Cartan form is fundamental in a differential geometric characterization of Lie groups.

The curvature of $\omega\in\Omega^1(G,\g)$ is given as $R = d\omega + \frac12\omega\wedge\omega\in\Omega^2(G,\g)$, and~(\ref{eq:mcstructure}) is therefore a flatness condition $R=0$. Taking $\theta = -\omega$ as a solder form, we compute the torsion form 
$\Theta = d\theta + \omega\wedge\theta = -\frac12 \omega\wedge\omega\in \Omega^2(G,\g)$. This yields
\[\Theta(X,Y) = -[\omega(X),\omega(Y)]_\g .\]
Therefore, the Maurer--Cartan form is flat with constant torsion. 
\begin{proposition}\label{prop:rightMC}
Given a Lie group $G$ and the inverse of the (left) MC form $\rho_*\colon \Omega^0(G,\g)\rightarrow \XG$ in~(\ref{eq:rhomap1}), then 
$\{\Omega^0(G,\g),-[\cdot,\cdot],\vartr\}$ is a post-Lie algebra, where
\begin{eqnarray*}
x\vartr y &:=& \rho_*(x)(y)\\
\ [x,y](g) &:=& [x(g),y(g)]_\g.
\end{eqnarray*}
The connection $x\vartr y$ of the left MC form is related to the connection of the right MC form $x\tr y$, see~(\ref{eq:triangle}), as
\begin{equation}x\vartr y = x\tr y + [x,y].\label{eq:tritri}\end{equation}
\end{proposition}
\begin{proof}From the right MC form we get the post-Lie algebra $\{\Omega^0(G,\g),[\cdot,\cdot],\tr\}$ (Proposition~\ref{prop:homogen}). The
correspondence between a right trivialized $\tilde{x}$ and the corresponding left trivialized $x\in\Omega^0(G,\g)$ is given as $x(g) = \omega(\lambda_*\tilde{x})(g) = \Ad_g\tilde{x}(g)$, 
from which~(\ref{eq:tritri}) follows by differentiation. Hence, by Proposition~\ref{prop:twist}, $\{\Omega^0(G,\g),-[\cdot,\cdot],\vartr\}$ is post-Lie.
\end{proof}

\begin{remark}Since $\rho_*\colon \Omega^0(G,\g)\rightarrow \XG$ is an isomorphism, we can equivalently define $\vartr\colon\XG\xpr\XG\rightarrow\XG$ as
$X\vartr Y = \rho_*\left(X(\omega(Y))\right)$. This is a flat Koszul connection on $\XG$ with torsion $[X,Y]=-\rho_*[\omega(X),\omega(Y)]_\g$. The vector fields on a Lie group with this connection and torsion is a prime example of a post-Lie algebra.
\end{remark}

\subsubsection{Homogeneous spaces, a leftist view} \label{sec:leftist}
In this section we return to a post-Lie structure related to vector fields on a homogeneous manifold, but unlike the conventional formulation of numerical integration, outlined in Section~\ref{sec:lgi}, we will here
discuss an alternative formulation which has hitherto not been applied in numerical algorithms. This formulation is better suited for using moving frames to  choose isotropy in numerical integration.

We recall some aspects of the Klein--Cartan geometry of a homogeneous space $\M$ from~\cite{sharpe1997dg}. Given a transitive left action $\cdot\colon \G\times \M\rightarrow \M$, and an arbitrary point $o\in \M$, let $H = \stset{h\in G}{h\dpr o =o}$ be the isotropy subgroup at $o$ with Lie algebra $\h<\g$. Define 
\begin{equation}\label{eq:atiyahbundle}G\times_{\! H}\g := \left(G\times \g\right)/\sim\ ,\end{equation}
where $(g,v)\sim (gh,\Ad_{h^{-1}} v)$ for all $h\in H$. Define the map
\begin{equation}\label{eq:rhomap}{\rho_{\! M}}\colon G\times_{\! H}\g\rightarrow \M,\quad (g,v)\mapsto g\exp(v)\cdot o , \end{equation}
and its differential with respect to the second variable ${\rho_{\! M}}_*\colon G\times_{\! H}\g\rightarrow \TM$ as
\begin{equation}\label{eq:rhostarmap}{\rho_{\! M}}_*(g,v):= \left.\frac{\partial}{\partial t}\right|_{t=0}g \exp(t v)\dpr o\in T_{g\cdot o}\M.\end{equation}
Since ${\rho_{\! M}}_*(g,v+v^\perp) = {\rho_{\! M}}_*(g,v)$ for all $v^\perp\in \h$, it follows that ${\rho_{\! M}}_*$ induces a smooth quotient mapping
\begin{equation}\overline{{\rho_{\! M}}}_* \colon G\times_{\! H}\g/\h\rightarrow \TM, (g,v+\h)\mapsto {\rho_{\! M}}_*(g,v),
\end{equation}
where $\g/\h$ denotes the quotient as vector spaces. In~\cite[Prop. 5.1]{sharpe1997dg} it is shown:

\begin{proposition} $\overline{{\rho_{\! M}}}_*$ defines an isomorphism $G\times_{\! H}\g/\h \simeq \TM$ as vector bundles over $\M$.
\end{proposition}

Thus, tangents in $T_{g\cdot o}\M$ are uniquely represented as $(g,v)\in G\times_{\! H}\g/\h$, while  finite motions are not invariant under change of isotropy; in general ${\rho_{\! M}}(g,v)\neq {\rho_{\! M}}(g,v+v^\perp)$  for $v^\perp\in \h$. In order to fix a choice of isotropy, it is useful to discuss the notion of \emph{gauges}. 

A \emph{Cartan gauge} is a  local section of the principal $H$-bundle $\pi\colon G\rightarrow \M, g\mapsto g\dpr o$, i.e.\ a map $\sigma\colon U\subset \M\rightarrow G$ such that
$\pi\opr\sigma = \Id$ on an open set $U$. We denote this $(\sigma,U)$.

The \emph{Darboux derivative} of a map $f\colon \M\rightarrow G$, denoted  $\omega_f \colon \TM\rightarrow \g$, is defined as 
the pullback of the Maurer--Cartan form $\omega$ on $G$ along $f$, 
\begin{equation}\label{eq:darboux}\omega_f := f^*\omega = \omega\opr f_*,\end{equation}
where $f_*\colon \TM\rightarrow TG$ is the differential of $f$. This will always satisfy the Cartan condition
\begin{equation}\label{eq:cartanflat}d\omega_f + \frac12\omega_f\wedge\omega_f = 0,\end{equation}
and thus it is a flat $\g$-valued one-form $\omega_f\in \Omega^1(\M,\g)$. We call $f$ the \emph{primitive} of $\omega_f$. A generalization of the
fundamental theorem of calculus~\cite{sharpe1997dg} states that, \emph{locally}, a one-form $\theta\in\Omega^1(\M,\g)$ has a primitive $f\colon\M\rightarrow G$ if and only if
$\theta$ is flat, i.e.\  iff $\theta$ satisfies~(\ref{eq:cartanflat}). Furthermore, the primitive is determined uniquely up to left multiplication by a constant $c\in G$: if $\omega_f=\omega_{\tilde{f}}$ then $\tilde{f} = c f$ for some constant  of integration $c\in G$.

The Darboux derivative of the Cartan gauge $\omega_\sigma\in \Omega^1(\M,\g)$ is called an \emph{infinitesimal Cartan gauge}. Since $\omega_{\sigma}$ is flat, we can uniquely recover $\sigma$ from $\omega_{\sigma}$ (the integration constant is found by the condition that $\sigma$ is a section).

\begin{proposition}A Cartan gauge $(\sigma,U)$ with derivative $\omega_\sigma$,  defines a map $\sigma_*\colon \TM\rightarrow G\times_{\! H}\g$ as
\begin{equation}\label{eq:sigmastar}\sigma_*(v) = \sigma(\pi v)\xpr_H \omega_{\sigma}(v),
\end{equation} 
which extends to $\overline{\sigma}_*\colon \TM\rightarrow G\times_{\! H}\g/\h$ by composition with the projection $\g\to \g/\h$
\begin{equation}\overline{\sigma}_*\colon \TM\stackrel{\sigma_*}{\longrightarrow}G\times_{\! H}\g\rightarrow G\times_{\! H}\g/\h .
\end{equation} 
The maps $\overline{\sigma}_*$ and $\overline{{\rho_{\! M}}}_*$ are locally inverse vector bundle isomorphisms, 
$\overline{\sigma}_*\opr\overline{{\rho_{\! M}}}_*=\Id$ and $\overline{{\rho_{\! M}}}_*\opr\overline{\sigma}_*=\Id$ on their domains of definition.
\end{proposition}

\begin{proof} This is a straightforward computation.
\end{proof}

Infinitesimally all gauges are equivalent, but how does a change of gauge affect the finite motions induced by ${\rho_{\! M}}$?
A \emph{retraction map} is a smooth, locally defined map ${\cal R}\colon\TM\rightarrow \M$ such that:
\begin{itemize}
\item ${\cal R}(v) = \pi_M v$ if and only if $v$ is a 0-tangent, where
$\pi_M\colon\TM\rightarrow \M$ is the natural projection.
\item ${\cal R}'(0)=\Id$ (the identity on the tangent fibre). 
\end{itemize}Retractions provide a useful way of formulating numerical integration schemes~\cite{celledoni2003implementation}.
For any Cartan gauge $(\sigma,U)$ there corresponds a retraction map ${\cal R}_\sigma\colon TU\rightarrow \M$ defined as
\begin{equation}{\cal R}_\sigma(v) = {\rho_{\! M}}\opr \sigma_*(v) = \sigma(\pi v)\exp(\omega_{\sigma}(v))\dpr o,\quad\mbox{for $v\in TU$.}\label{eq:retraction}
\end{equation}
If $(\sigma,U)$ and $(\sigma',U)$ are two gauges on $U\subset\M$, then $\sigma'(u) = \sigma(u)h(u)$ for some smooth $h\colon U\rightarrow H$.
The corresponding infinitesimal gauges $\omega_{\sigma}$ and $\omega_{\sigma'}$ are related as~\cite[p.\ 168]{sharpe1997dg}
\begin{equation}
\omega_{\sigma}' = \Ad_{h^{-1}}\omega_{\sigma} + \omega_h,
\end{equation}
where $\omega_h\in \Omega^1(U,\h)$ is the Darboux derivative of $h$. Thus, the modified retraction is given as
 \begin{equation}{\cal R}_{\sigma'} = \sigma(\pi v)\exp\left(\omega_{\sigma}(v)\! +\!\omega_h(v)\right)\dpr o.
 \end{equation}
 If $h$ is constant then $\omega_h=0$ and the two retractions are equal, but generally they differ. 
 Numerical algorithms are built from finite motions where the choice of gauge (isotropy) influences the numerical methods. This has been discussed in~\cite{berland2002isotropy,lewis2002gia}. However, we believe that there is still much work to be done to better understand this aspect of geometric numerical integration. In the sequel we will discuss choice of gauge using moving frames.

Let $\Gamma(G\times_{\! H}\g)$ denote sections of the vector bundle $G\times_{\! H}\g\rightarrow \M$. This space can be identified 
with the subspace $\Omega_H^0(G,\g)\subset\Omega^0(G,\g)$ defined as
\begin{equation}\label{eq:omegaH}\Omega^0_H(G,\g) := \stset{x\in \Omega^0(G,\g)}{x(gh) = \Ad_{h^{-1}}x(g)}.\end{equation}
\begin{proposition}\label{prop:11cor}There is a 1--1 correspondence between elements $x\in \Omega^0_H(G,\g)$ and elements $X\in\Gamma(G\times_{\! H}\g)$,
defined locally by a Cartan gauge $(\sigma,U)$ as
\[X(u) =\sigma(u)\times_{\! H}x(\sigma(u)),\quad \mbox{for $u\in U\subset \M$.}\]
The map $x\mapsto X$ is independent of the choice of gauge $\sigma$. 
\end{proposition}
\begin{proof}The map $x\mapsto X$ is independent of $\sigma$ since $\sigma(u)h\times x(\sigma(u)h)\sim \sigma(u)\times \Ad_h x(\sigma(u)h) = 
\sigma(u)\times x(\sigma(u))$. Given $X$ we recover $x$ on $\sigma(U)\subset G$, and reconstruct $x$ on the fibres of $H\rightarrow G\rightarrow\M$
by $x(gh) = \Ad_{h^{-1}}x(g)$.
\end{proof}

Thus, to sum up, we have an identification $\Omega^0_H(G,\g)\simeq \Gamma(G\times_{\! H}\g)$. An $X\in \Gamma(G\times_{\! H}\g)$ corresponds to a
vector field ${\rho_{\! M}}_*\opr X\in \XM$ and a diffeomorphism ${\rho_{\! M}}\opr X\in \mbox{Diff}(\M)$. Furthermore, given a Cartan gauge $(\sigma,U)$, we
can map a vector field $Y\in \XM$ to $\sigma_*\opr Y\in \Gamma(G\times_{\! H}\g)$, but this map depends on the gauge $\sigma$.
Finally, there is a post-Lie algebra also associated with this view of homogeneous spaces:

\begin{proposition}\label{prop:atiyahpostlie} $\{\Omega_H^0(G,\g),-[\cdot,\cdot],\vartr\}$, with $\vartr$ and $[\cdot,\cdot]$ defined in Prop.~\ref{prop:rightMC},  is post-Lie.
\end{proposition}
\begin{proof}A straightforward computation reveals that $(x\vartr y)(gh) = \Ad_{h^{-1}} (x\vartr y) (g)$ and\\ $[x(gh),y(gh)] = \Ad_{h^{-1}}[x(g),y(g)]$.
Hence $\vartr$ and $[\cdot,\cdot]$ are well defined on the subspace $\Omega^0_H(G,\g)$.
\end{proof}

What is the parallel transport in this post-Lie algebra? Note that  $x\in\Omega^0_H(G,\g)$ defines a vector field ${\rho_{\! M}}_*x\in \XG$ with a flow $\Phi_t\colon G\rightarrow G$ satisfying 
\begin{equation}\Phi_t(gh) = \Phi_t(g)h,\quad \mbox{for all $h\in H$. }\label{eq:flow1}\end{equation}
Define parallel transport along the flow $\Phi_t^*\colon \Omega^0_H(G,\g)\rightarrow \Omega^0_H(G,\g)$ as
\begin{equation}\Phi_t^*y(g) := y(\Phi_t(g)).\label{eq:ptdef}
\end{equation}

\begin{proposition}
Parallel transport in $\Omega^0_H(G,\g)$ is given as 
\begin{equation}\Phi^* y=\exp(x\vartr)y=\exp(x\vartr)y=y+x\vartr y + \frac12(x\vartr(x\vartr y))+\frac16(x\vartr(x\vartr(x\vartr y)))+\cdots .\label{eq:parallel}\end{equation}
where $\Phi^*\colon G\rightarrow G$ is the $t=1$ flow of ${\rho_{\! M}}_*x\in \XG$.
\end{proposition}

\begin{proof}This follows from the Taylor expansion of~(\ref{eq:ptdef}).
\end{proof}

\begin{remark}The post-Lie structure captures algebraically both infinitesimal aspects of  homogeneous spaces and also finite motions such as flows and parallel transport. It is therefore clear that the post-Lie structure cannot carry over to the quotient space $\Gamma\left(G\times_{\! H}\g/\h\right)\simeq\XM$: both $\vartr$ and $[\cdot,\cdot]$ change under an isotropy change $x\mapsto x+x^\perp$ where $x,x^\perp\in \Omega^0_H(G,\g)$ and  $x^\perp(g)\in \h$. Hence this post-Lie structure does not carry over to $\XM$, except when the action is free. 
\end{remark}

\begin{remark} 
Most of the work on numerical integration on homogeneous spaces over the last
15 years follow the formulation of~\cite{munthe-kaas1997nio}, which can be cast into the post-Lie algebra structure of the right MC form discussed in Section~\ref{sec:lgi}.
In the present section, we have detailed an alternative post-Lie algebra structure on $\M$, derived from the left MC form. Combined with moving frame algorithms for choosing the gauge, we believe that this presents an interesting new alternative formulation, which may incorporate the geometry of the underlying problem into the choice of gauge (isotropy) in numerical integration schemes. 
\end{remark}

\begin{remark}The vector bundle $G\times_H\g\rightarrow\M$ with the  "anchor map" ${\rho_{\! M}}_*\colon G\times_H\g\rightarrow T\M$ is a mathematical structure known as the
 \emph{Atiyah Lie algebroid}~\cite{mackenzie2005general}. We have shown that there is a natural post-Lie algebra structure associated with this structure.
This motivates a general definition of a \emph{post-Lie algebroid}.
\end{remark}

\subsubsection{Post-Lie algebroids}\label{sec:postliealgebroid} All the examples above fit into a common structure we will discuss in the language of linear connections on 
Lie algebroids~\cite{degeratulinear,mackenzie2005general}. 
 
 \begin{definition} A \emph{Lie algebroid} is a triple $\{E,\llbracket\cdot,\cdot\rrbracket_E,\rho_*\}$ consisting of a vector bundle $E\rightarrow \M$, a Lie bracket $\llbracket\cdot,\cdot\rrbracket_E$ on the module of sections
 $\Gamma(E)$ and a morphism of vector bundles $\rho_*\colon E\rightarrow\TM$ called the \emph{anchor}. The anchor and the bracket must satisfy the Leibniz rule
  \begin{equation}\llbracket x,\phi y\rrbracket_E = \phi \llbracket x,y\rrbracket_E + (\rho_*\opr x)(\phi) y, \label{eq:laleib}
 \end{equation}
 for all $x,y\in \Gamma(E)$ and all $\phi\in C^\infty(\M,\RR)$, where $(\rho_*\opr x)(\phi)$ denotes the Lie derivative of $\phi$ along of the vector field $\rho_*\opr x\in \XM$.
 \end{definition}
 
 The definition implies that the anchor $\rho_*$ is a Lie algebra morphism, sending the bracket $\llbracket\cdot,\cdot\rrbracket_E$ on $\Gamma(E)$ to the Jacobi--Lie bracket $\llbracket\cdot,\cdot\rrbracket_J$ on $\XM$.
 
 A \emph{linear connection} on a Lie algebroid is defined as a bilinear product $\tr\colon \Gamma(E)\times \Gamma(E)\rightarrow\Gamma(E)$
 such that
 \begin{eqnarray}
 (\phi x)\tr y & = & \phi (x\tr y)\\
 x\tr (\phi y) & = & (\rho_*\opr x)(\phi) y
 \end{eqnarray}
 for all $x,y\in \Gamma(E)$ and $\phi\in C^\infty(\M,\RR)$. A Lie bracket $[\cdot,\cdot]$ on $\Gamma(E)$ is called \emph{tensorial} if it is $C^\infty(\M,\RR)$-linear, i.e.\ 
 \begin{equation}\label{eq:tensorial}[x,\phi y] = \phi [x,y] = [\phi x,y], 
 \end{equation}
 for all $x,y\in \Gamma(E)$, $\phi\in C^\infty(\M,\RR)$. The following definition is novel.
 
 \begin{definition} A \emph{post-Lie algebroid} is a four-tuple $\{E,[\cdot,\cdot],\tr,\rho_*\}$, where $E\rightarrow \M$ is a vector bundle, $\rho_*\colon E\rightarrow \TM$ is a morphism of vector bundles, 
 $[\cdot,\cdot]$ is a tensorial Lie bracket on $\Gamma(E)$ and $\tr$ is a linear connection on $\Gamma(E)$ such that $\{\Gamma(E),[\cdot,\cdot],\tr\}$ is a post-Lie algebra.
 \end{definition}
 
 \begin{proposition} Let $\{E,[\cdot,\cdot],\tr,\rho_*\}$ be a post-Lie algebroid. Define the Lie bracket $\llbracket x,y\rrbracket_E:= x\tr y-y\tr x + [x,y]$ on $\Gamma(E)$, then $\{E,\llbracket\cdot,\cdot\rrbracket_E,\rho_*\}$
 is a Lie-algebroid.
 \end{proposition}
 
 \begin{proof} Proposition~\ref{prop:jlbracket} shows that $\llbracket\cdot,\cdot\rrbracket_E$ is a Lie bracket. The Leibniz rule~(\ref{eq:laleib}) is easily verified.
 \end{proof}
 
 The definition of torsion~(\ref{eq:torsion}) and curvature~(\ref{eq:curvature}) is valid also for linear connections on a Lie algebroid, with $\llbracket\cdot,\cdot\rrbracket_J$ replaced by $\llbracket\cdot,\cdot\rrbracket_E$.
 
 \begin{proposition} Let $\{E,\llbracket\cdot,\cdot\rrbracket_E,\rho_*\}$ be a Lie algebroid and $\tr$ a linear connection on $\Gamma(E)$ with zero curvature $R=0$ and constant torsion $\nabla T=0$. Let
 $[x,y] = -T(x,y) = \llbracket x,y\rrbracket_E - x\tr y +y\tr x$. Then  $\{E,[\cdot,\cdot],\tr,\rho_*\}$ is a post-Lie algebroid.
 \end{proposition}
 
 \begin{proof} $R=0$ and $\nabla T=0$ imply the post-Lie relations~(\ref{eq:postlie1})-(\ref{eq:postlie2}). A straightforward computation shows that $[\cdot,\cdot]$ satisfies~(\ref{eq:tensorial}), so it is tensorial.
 The Jacobi rule $\csum([[x,y],z])=0$ is verified by a lengthy computation.
 \end{proof}

\subsection{Moving frames} Cartan's method of moving frames provides an important tool for choosing gauges that are naturally derived from the geometry of e.g.\ a differential equation or other geometric objects such as curves and surfaces in a homogeneous space (Klein geometry).  Peter Olver and his co-workers have developed  moving frames into a powerful  tool in applied and computational mathematics~\cite{mansfield2010apg,olver2005aso}. See also~\cite{gardner1989method} for applications of moving frames in control theory.

Let $\M$ be a manifold and $G$ a Lie group, with algebra $\g$, that acts on $\M$ on the left. We do not require the action to be transitive, so $\M$ needs not be a homogeneous space.

\begin{definition}A left moving frame is a map $\sigma\colon \M\rightarrow G$ such that
\[\sigma(g\cdot u) = g\sigma(u)\quad\mbox{for all $g\in G$ and $u\in\M$,}\]
a right moving frame is a map $r\colon \M\rightarrow G$ such that
\[r(g\cdot u) = r(u)g^{-1}\quad\mbox{for all $g\in G$ and $u\in\M$.}\]         
\end{definition}
If $r$ is a right moving frame then $\sigma(u) = r(u)^{-1}$ (inverse in $G$) is a left moving frame. 
Moving frames exist if and only if the $G$ action on $\M$ is free and regular. In that case, moving frames can be constructed (locally) as follows~\cite{olver2005aso}:
\begin{enumerate}
\item Choose a submanifold $\K\subset\M$ which is transverse to the $G$ orbits and of the maximal dimension $p=\dim(\M)-\dim(G)$. Locally, there is one point in $\K$ for each orbit, and each
orbit intersect $\K$ in one point. In coordinates, $\K$ is often chosen by setting $d=\dim(G)$ of the coordinates to constant values. 
\item A right moving frame $r$ is found by solving the normalization equations
$r(u)u \in \K$ for $r(u)$. 
\item A left moving frame is obtained by inverting $r$.
\end{enumerate}

If the action is not free, there is a standard procedure of obtaining a free and regular action  by \emph{prolongation} of the group action, i.e.\ we extend $\M$ to a jet-space $J^k(\M)$, which is the geometrical way of saying that we consider the space of all curves in $\M$ represented by Taylor expansions up to order $k$. The prolongation of the group action is the natural induced action of $G$ on (Taylor expansions of) curves. Coordinates on the jet-space are given by the (higher order)  derivatives of curves. We can always obtain a free regular action (and thus a moving frame) by prolongation. Thus, by this construction we find a left moving frame
$\sigma\colon J^k(\M)\rightarrow G$.

Moving frames are closely related to Cartan gauges on a homogeneous space $\M$. Let $G$ act transitively on $\M$ with an isotropy  subgroup $H$, and
let $\sigma\colon\M\rightarrow\G$ be a local section of the bundle $\pi\colon \G\rightarrow\M$. Note that $\sigma$ is a section if and only if
\begin{equation}\sigma(g\dpr u) = g\sigma(u)h(u)\quad\mbox{for  $h\colon \M\rightarrow H$.}
\end{equation}
Thus it is a left moving frame, up to isotropy.
Such a map is also called a \emph{partial}  moving frame~\cite{lewis05cfg}. 

Thus, the theory of moving frames (full and partial) provides geometric ways of constructing sections $(\sigma,U)$ of $G\rightarrow\M$, hence also
 geometric ways of fixing isotropy through the map $\sigma_*\colon \TM\rightarrow  \Gamma(G\times_{\! H}\g)\simeq \Omega^0_H(G,\g)$ in~(\ref{eq:sigmastar}).
On $\Omega^0_H(G,\g)$ we have all the tools we need to do numerical integration and analysis of numerical integration schemes. The details of such algorithms are
subject to future research. We see at least two useful ways to proceed in choosing $\sigma$.
\begin{itemize}
\item By prolongation of the group action we can obtain a full moving frame $\sigma\colon J^k(\M)\rightarrow G$. To solve a differential equation $u'=F(u)$, $F\in\XM$, we must also prolong $F$ to the jet-bundle. This should be a very attractive numerical method in cases where we can compute the $k$-th derivatives of $F$, either by computer algebraic means, or by automatic differentiation.
\item In the case where $\M$ is a symmetric space, there is a canonical choice of a section $\sigma\colon \M\rightarrow G$. In this case there exists a canonical splitting $\g = \h\oplus \kk$, where $\h$ is a subalgebra and $\kk$ is a Lie triple system (LTS), see~\cite{loos1969symmetric}. The infinitesimal gauge $\omega_\sigma$ takes values in $\kk$, and thus exponentials need only be computed on the LTS. Efficient algorithms for computing exponentials on an LTS are discussed in~\cite{zanna2002generalized}. This theory opens up the possibility of new classes of numerical integration on symmetric spaces.
\end{itemize}

\section{The algebraic structure of post-Lie and D-algebras}
In this section we discuss the algebraic structure of general post-Lie algebras $\{\A,[\cdot,\cdot],\tr\}$. Various aspects of this theory can also be found in~\cite{lundervold2011bea,lundervold2009hao,munthe-kaas2003oep,munthe-kaas2008oth}. Unlike previous work, we here develop the core theory from the axiomatic definition of a post-Lie algebra, and establish the functorial relationship between post-Lie algebras and their enveloping D-algebras. Moreover, we find that a magmatic view of planar trees and forests gives rise to new recursive formulas for various algebraic operations, which simplify computer implementations.

\subsection{Free post-Lie algebras}
In \cite{chapoton2001pla} Chapoton and Livernet gave an explicit description of the free pre-Lie algebra in terms of decorated rooted trees and grafting. In this section we will see that there is a similar description of the free post-Lie algebra. In fact, we will show that the free post-Lie algebra can be described as the free Lie algebra over planar rooted trees, extended with a connection given by left grafting of trees. Furthermore, we will relate post-Lie algebras to D-algebras, studied in connection with numerical Lie group integration~\cite{ lundervold2009hao,munthe-kaas2008oth}. The universal enveloping algebra of a post-Lie algebra is a D-algebra, and the post-Lie algebra is recovered as the derivations in the D-algebra.


\paragraph{Trees.} The algebraic definition of a \emph{magma} is a set $\C$ with a binary operation $\star$ without any algebraic relations imposed. The free magma over $\C$ consists of all possible ways to parenthesize binary operations on $\C$. There are several isomorphic ways of representing a free magma in terms of trees: as binary trees with colored leaves or, as we will do in the sequel, as planar trees with colored nodes. See~\cite{ebrahimi2012magnus} for an isomorphism between these representations.
The set $\C$ will henceforth be called the set of \emph{colors}. We let $\OTC$ be the set of all planar (or ordered)\footnote{Trees with different orderings of the branches are considered different, as when pictured in the plane.} rooted trees with nodes colored by $\C$. Formally, we define it as the free magma 
\[\OTC := \operatorname{Magma}(\C).\]

On trees we interpret $\star$ as the \emph{Butcher product}~\cite{butcher1972aat}: $\tau_1\star\tau_2=\tau$ is a tree where the root of the tree $\tau_1$ is attached to the left part of the root of the tree $\tau_2$. For example: 
\[\AabB \star \aaababbb = \, \aAabBaababbb \ \ =\  (\ab\star\AB)\star((\ab\star(\ab\star\ab))\star\ab) .\]
If $\C=\{\ab\}$ has only one element, we write $\OT:=\OT_{\{\ab\}}$. The first few elements of $\OT$ are:
\[\OT = \left\{\ab, \aabb,\aaabbb, \aababb, \aaaabbbb,\aaababbb,\aaabbabb,
 \aabaabbb, \aabababb,\ldots
\right\}.\]
Note that any $\tau\in\OTC$ has a unique maximal right factorization 
\[\tau = \tau_1\star(\tau_2\star(\cdots(\tau_k\star c))),\quad \mbox{where $c\in \C$ and $\tau_1,\ldots,\tau_k\in \OTC$.}\]
Here $c$ is the root, $k$ is the \emph{fertility} of the root and  $\tau_1,\ldots,\tau_k$ are the {branches} of the root. Let $\k$ be a field of characteristic zero and write $\kOTC$ for the free $\k$-vector space over the set $\OTC$, i.e.\ all $\k$-linear combinations of  trees. We define \emph{left grafting}\footnote{Various notations for similar grafting products are found in the literature, e.g.\ $u\tr v = u[v] =u\curvearrowright v$.}
 $\tr\colon \OTC\xpr\OTC\rightarrow \kOTC$ by the recursion
\begin{equation}\label{eq:graft} 
\begin{split}
\tau \tr c &:=  \,\,\tau\star c, \quad\mbox{for $c\in \C$}\\
\tau\tr (\tau_1\star(\tau_2\star(\cdots(\tau_k\star c))))  &:=\,\,  
\tau\star(\tau_1\star(\tau_2\star(\cdots(\tau_k\star c))))\\ 
&\,\,+(\tau\tr \tau_1)\star (\tau_2\star(\cdots(\tau_k\star c)))\\ 
&\,\,+\tau_1\star((\tau\tr\tau_2)\star(\cdots(\tau_k\star c)))\\
&\,\,+\cdots  \\
&\,\,+\tau_1\star(\tau_2\star(\cdots((\tau\tr\tau_k)\star c))) .
\end{split}
\end{equation}
Thus $\tau_1\tr\tau_2$ is the sum of all the trees resulting from attaching the root of $\tau_1$ to all the nodes of the tree $\tau_2$ from the left. Example:
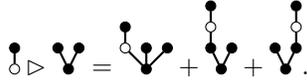
\begin{equation*}
  \AabB\tr\aababb=\aAabBababb +\aaAabBbabb + \aabaAabBbb.
\end{equation*}

%
%
%

\paragraph{Free Lie algebras of trees.}
Let $\g=\fla(\OTC)$ denote the free Lie algebra over the set $\OTC$ \cite{reutenauer93fla}. For $\C=\{\ab\}$, a Lyndon basis is given up to order four as~\cite{munthe-kaas2003oep}:
\[\fla(\OTC) = \k\left\{\ab, \aabb,\aaabbb, \aababb, \left[\aabb,\ab\right], \aaaabbbb,\aaababbb,\aaabbabb,
\left[ \aaabbb,\ab \right], \aabaabbb, \aabababb,\left[\aababb,\ab\right], \left[\left[\aabb,\ab\right],\ab\right],\ldots
\right\}.
\]

\begin{proposition}\label{prop:fct}
Let the free Lie algebra $\g = \fla(\OTC)$ be equipped with a product $\tr\colon \g\times\g\rightarrow \g$, extended from the left grafting defined on $\OTC$ in (\ref{eq:graft}) as
\begin{eqnarray}
u\tr [v,w] & = & [u\tr v,w] + [v,u\tr w]\label{eq:tr1}\\
\left[u,v\right]\tr w & = &  a(u,v,w)-a(v,u,w) \label{eq:tr2}
\end{eqnarray}
for all $u,v,w \in \g$. Then $\{\fla(\OTC),[\cdot,\cdot],\tr\}$ is post-Lie.
\end{proposition}

\begin{proof}Since any $u,v,w\in \g$ can be written as a sum of trees and commutators of trees, the connection is well-defined on $\g$. It satisfies the axioms of a post-Lie algebra by construction.
\end{proof}

%

\paragraph{Free post-Lie algebras.}Proposition~\ref{prop:fct} shows that the free Lie algebra of ordered trees has naturally the structure of a post-Lie algebra  $\fFCT(\C) := \{\fla(\OTC),[\cdot,\cdot],\tr\}$. We call this the  \emph{free post-Lie algebra} over the set $\C$ for the following reason:

\begin{theorem}\label{th:freeFCT} For any post-Lie algebra $\{\A,[\cdot,\cdot],\tr\}$ and any
function $f\colon \C\rightarrow \A$, there exists a unique morphism of post-Lie algebras $\F\colon \fFCT(\C)\rightarrow \A$ such that $\F(c) = f(c)$ for all $c\in \C$.
\end{theorem}

\begin{proof} We construct $\F$ in two stages. First we show, using $\tr$, that $f$ extends uniquely to a function $\F_{\OTC}\colon \OTC \rightarrow\A$. Then by universality of the free Lie algebra, there is a unique Lie algebra homomorphism $\F\colon\fla(\OTC)\rightarrow\A$. We show that this is also a homomorphism for the connection product $\tr$.

To construct the extension to $\OTC$ we first observe that the magmatic product  $\tau\star\tau'$ on $\OTC$ (the Butcher product of two trees) can be expressed in terms of left grafting $\tr$. This is done by induction on the fertility of $\tau'$. For fertility 0, i.e.\ $\tau'=c\in \C$, we have $\tau\star c = \tau\tr c$.
For fertility $k$ we write $\tau' =  \tau_1\star(\tau_2\star(\cdots(\tau_k\star c)))$ and find from~(\ref{eq:graft})
\[\tau\star \tau' = \tau\tr\tau' - (\tau\tr\tau_1)\star(\tau_2\star(\cdots(\tau_k\star c)))
-\cdots -  (\tau_1\star(\tau_2\star(\cdots(\tau\tr\tau_k\star c))) .\]
In the right hand side of the equation, the fertility of any term to the right of a $\star$-product is smaller than $k$, which completes the induction. The fact that $\OTC$ is freely generated from $\C$ by the product $\star$ ensures that $\F_{\OTC}$ is uniquely defined by
\begin{eqnarray*}\F_{\OTC}(c) = f(c)\quad \mbox{for all $c\in \C$}\\
\F_{\OTC}(\tau\tr\tau') = \F_{\OTC}(\tau)\tr\F_{\OTC}(\tau'),
\end{eqnarray*}
and hence that $\F\colon\fla(\OTC)\rightarrow\A$ is uniquely defined as a Lie algebra homomorphism.

Finally, by induction on the length of iterated commutators, we see that $\F(u\tr v) = \F(u)\tr\F(v)$ for all $u,v\in \fla(\OTC)$: If $u,v\in \OTC$ this holds by construction. Assuming that $\F(u\tr v) = \F(u)\tr\F(v)$ whenever $u$ and $v$ are iterated commutators of length at most $k$, we find by using~(\ref{eq:tr1})--(\ref{eq:tr2}) that $\F([u,\tau_1]\tr [v,\tau_2]) = \F([u,\tau_1])\tr \F([v,\tau_2])$ for all $\tau_1,\tau_2\in \OTC$.
\end{proof}

\begin{proposition}Let $\fFCT(\C)$ be graded by the number $n$ counting the number of nodes in the trees. Then 
\[\mbox{dim}(\fFCT(\C)_n) = \frac{1}{2n}\sum_{d|n}\mu(\frac{n}{d})\binom{2d}{d} n^{|\C|},\]
where $\mu$ is the M{\"o}bius function. For $|\C|=1$ the dimensions are $1, 1, 3, 8, 25, 75, 245,\ldots$, see \cite{oeisA022553}.
\end{proposition}

\begin{proof}See~\cite{munthe-kaas1999cia} and~\cite{munthe-kaas2003oep}.
\end{proof}

\begin{remark}The same dimensions also appear for the primitive Lie algebra of the Hopf algebra CQSym (Catalan Quasi-Symmetric functions)~\cite{novelli2007parking}. 
\end{remark}

\subsection{Universal enveloping algebras}\label{sec:Dalg}
\paragraph{D-algebras.}
In Section~\ref{sec:LBseries} we will describe certain algebraic structures that occur naturally in the study of numerical integration methods on manifolds \cite{munthe-kaas2008oth}. Central to this work are algebras of derivations, called $D$-algebras. We will see that post-Lie algebras relate to D-algebras similarly to how Lie algebras relate to their universal enveloping algebras.

\begin{definition}[D-algebra \cite{munthe-kaas2008oth}]\label{Dalg}
Let $B$ be a unital associative algebra with product $u,v \mapsto uv$, unit $\one$ and equipped with a non-associative product $\cdot\tr\cdot \colon B \otimes B \rightarrow B$ such that $\one\tr v = v$ for all $v\in B$. Write $\Der(B)$ for the set of all $u \in B$ such that $u\tr \cdot$ is a derivation:
$$\Der(B) = \{u\in B \,\,|\,\, u\tr(vw) = (u\tr v)w + v(u\tr w) \hspace{0.2cm} \text{for all } v,w \in B\}.$$  $B$ is called a \emph{D-algebra} if the product $u,v\mapsto uv$ generates $B$ from $\{\one,\Der(B)\}$ and, furthermore, for any $u \in\Der(B)$ and any $v,w\in B$ we have
\begin{eqnarray}
v\tr u &\in& \Der(B)\\
(uv)\tr w & =&    u\tr(v\tr w) - (u\tr v)\tr w.\label{eq:assocB}
\end{eqnarray}
\end{definition}

\begin{proposition}
 If $B$ is a $D$-algebra then the derivations $\Der(B)$ form a post-Lie algebra, with torsion $[u,v] = uv-vu$ and connection $\tr$.
 \end{proposition}
 
 \begin{proof}If $u,v\in \Der(B)$ we note that
 \[(uv-vu)\tr\cdot = u\tr(v\tr\cdot)-  v\tr(u\tr\cdot) +(u\tr v)\tr\cdot - (v\tr u)\tr\cdot .\]
 The first two terms on the right is a commutator of two derivations and is therefore a derivation. The last two terms are derivations separately. Hence, $[u,v]\in \Der(B)$ and $\{\Der(B),[\cdot,\cdot]\}$ is a Lie algebra. The other axioms of being post-Lie follows easily from the definition of a
 D-algebra.
 \end{proof}

\paragraph{Universal enveloping algebras.} Let $\{A,[\cdot,\cdot],\tr\}$ be a post-Lie algebra, and let $\{U(A),\ass\}$ be the universal enveloping algebra of the Lie algebra $\{A, [\cdot, \cdot]\}$, where $\ass$ denotes the unital associative product $u,v\mapsto uv$ in $U(A)$. By the Poincar\'{e}--Birkhoff--Witt
 theorem we can embed $A$ as a linear subspace of $U(A)$, 
such that $[u,v]=uv-vu$. The embedding of $A$ is also denoted by $A$. The product $\tr$ on $A$ can be extended to $U(A)$:
\begin{eqnarray}
\label{eq:ex0} \one\tr v & = & v\\
\label{eq:ex1} u \tr (vw)& =& (u \tr v)w + v(u\tr w)\\
\label{eq:ex2} (uv) \tr w & =&  u\tr (v\tr w) - (u\tr v)\tr w,
\end{eqnarray}
for all $u\in A$ and $v,w\in U(A)$.
\begin{proposition} Equations~(\ref{eq:ex0})--(\ref{eq:ex2}) define a unique extension of $\tr$ from $A$ to $U(A)$. With the non-associative product $\tr$, $\{U(A),\ass,\tr\}$ is a $D$-algebra, with $A\subset\Der(U(A))$.
\end{proposition}
\begin{proof}See~\cite[Theorem V.1]{jacobson1979la} for a proof that a derivation on a Lie algebra $A$ extends uniquely to a derivation on $U(A)$. This justifies the extension on the right~(\ref{eq:ex1}). The extension on the left, given by (\ref{eq:ex0}) and (\ref{eq:ex2}), is compatible with the the embedding $[u,v]\mapsto uv-vu$ due to the flatness condition (\ref{eq:postlie2}) for post-Lie algebras. From the PBW basis on $U(A)$ it follows that these equations extend $\tr$ uniquely to all of $U(A)$ also on the left. From~(\ref{eq:ex1}) we see that $A\subset \Der(U(A))$. 
\end{proof}

\begin{definition}[Universal enveloping algebras]
We call $\{U(A),\ass,\tr\}$ the \emph{universal enveloping algebra} of the \emph{post-Lie} algebra $A$. 
\end{definition}

A D-algebra morphism is a linear map between D-algebras $\F\colon B'\rightarrow B$ such that $\F(\one) = \one$, $\F(uv) = \F(u)\F(v)$ and $\F(u\tr v) = \F(u)\tr\F(v)$.
Obviously $\F$ restricts to a post-Lie morphism $\Der(\F)\colon \Der(B')\rightarrow \Der(B)$ by $\Der(\F)([u,v]) = \F(uv-vu)$ and $\Der(\F)(u\tr v) = \F(u\tr v)$.
The following result, which is very similar to the corresponding result for Lie algebras, justifies naming $\{U(A),\ass,\tr\}$ the universal enveloping algebra of the post-Lie algebra $A$.

\begin{proposition}\label{prop:dalgUni}Let $A$ be post-Lie and $\iota\colon A\hookrightarrow \Der(U(A))$ the inclusion. For any D-algebra $B$ and any post-Lie morphism $f\colon A\rightarrow \Der(B)$ there exists a unique D-algebra morphism $\F\colon U(A)\rightarrow B$ such that $\Der(\F)\opr \iota  =  f$. 
\end{proposition}

\begin{proof}$\F$ is uniquely defined as a unital associative algebra morphism. It remains to verify that $\F(u\tr v) = \F(u)\tr \F(v)$. $U(A)$ can be graded by the length of the monomial basis of PBW. Using~(\ref{eq:ex0})--(\ref{eq:ex2}), it follows by induction on the grading that $\F(u\tr v) = \F(u)\tr \F(v)$.
\end{proof}

\begin{remark}
The preceding results establish a pair of adjoint functors between the categories of D-algebras and post-Lie algebras:
\begin{diagram}
U(\cdot):  \text{post-Lie}: &\pile{\rTo \\ \lTo}&\text{D-alg} :\Der(\cdot).
\end{diagram}
In other words, there is a natural isomorphism
\[\Hom_{\operatorname{postLie}}(\Der(A), B) \rightarrow \Hom_{\operatorname{D}}(A, U(B)).\]
%
\end{remark}

\paragraph{Free D-algebras.} A direct consequence of Theorem~\ref{th:freeFCT} and Proposition~\ref{prop:dalgUni} is the following characterization of a free D-algebra:
\begin{corollary}[{\cite[Proposition 1]{munthe-kaas2008oth}}]\label{univDalg}
The algebra $\D_{\C} := U(\fFCT(\C))$ is the free D-algebra over the set $\C$. That is, for any D-algebra $B$ and any function $f\colon \C\rightarrow \Der(B)$ there exists a unique D-algebra morphism $\F\colon \D_{\C}\rightarrow B$ such that $\F(c) = f(c)$ for all $c\in\C$.
\end{corollary}
The unital associative algebra of $\D_{\C}$ is  $U(\fla(\OTC))$, 
which by the Cartier--Milner--Moore theorem is the free associative algebra over $\OTC$. I.e.\ it is the noncommutative polynomials over rooted trees: $\D_{\C} = \k \langle \OT_{\C} \rangle= \k\{\OFC\}$, where 
$\k\{\OFC\}$ denotes the free vector space over the set of \emph{ordered forests}.
$\OFC := \OTC^*$ consist of all words of finite length over the alphabet $\OTC$, including the empty word $\one$. For $\C=\{\ab\}$, these are
\[\OF = \left\{\one, \,\, \ab, \,\, \ab\,\ab, \,\, \aabb, \,\, \ab\,\ab\,\ab, \,\, \aabb\,\ab, \,\, \ab\,\aabb, \,\, \aababb, \,\, \aaabbb, \,\, \cdots\right\}.\]   
We can create a tree from a forest $\omega$ by applying the operator $\Bplus_c: \OFC \rightarrow \OTC$, attaching the trees in $\omega$ onto a common root labelled by $c\in \C$, and we can create a forest from a tree using the operator $\Bminus: \OTC \rightarrow \OFC$ removing the root. The concatenation product $\omega_1,\omega_2\mapsto \omega_1\omega_2$ is the associative operation of sticking shorter words together to create longer words. 

To summarize, the free D-algebra $\D_\C$ is the vector space of forests $\k\{F_\C\}$ with unit $\one$, concatenation product and the left grafting product $\tr$ defined on trees in~(\ref {eq:graft}) and extended to forests by~(\ref{eq:ex0})--(\ref{eq:ex2}).
This free D-algebra carries a Hopf algebra structure, closely related to the Butcher-Connes--Kreimer Hopf algebra, to be discussed in the sequel.




\paragraph{The composition product $\opr$ on D-algebras.}
A \emph{dipterous} algebra \cite{loday2010cha} is a triple $\{B,\opr,\tr\}$, where $B$ is a vector space and $\opr$ and $\tr$ are two binary operations on $B$ satisfying:
  \begin{eqnarray}
  x\opr(y\opr z) & = & (x\opr y)\opr z
 \label{eq:dipt1}\\
 x\tr(y\tr z) &= &(x\opr y)\tr z\label{eq:dipt2}
 \end{eqnarray}
for all $x,y,z\in B$. Let $B$ be a D-algebra with concatenation $x,y\mapsto xy$ and connection product $x\tr y$. Define a product $\opr\colon B\times B\rightarrow B$ as 
\begin{equation}\label{eq:composition} 
\begin{split}
\one\opr y &= y\\
x\opr y &:=  \,\, xy + x\tr y\\
(xy)\opr z  &:=\,\, x\opr (y\opr z) - (x\tr y)\opr z\quad \mbox{for all $x\in \Der(B), y,z\in B$.}  
\end{split}
\end{equation}

\begin{proposition}If $B$ is a D-algebra then $\{B,\opr,\tr\}$ is a dipterous algebra.
\end{proposition}

\begin{proof}Proof by induction on the grading on $B$ provided by the PBW basis.
\end{proof}
\noindent The product $x,y\mapsto x\opr y$ will be referred to as the \emph{composition product}, while $x,y\mapsto xy$ is called either concatenation or \emph{frozen composition}, due to its interpretation for differential operators on manifolds. Let $A=\Omega^0(\M,\g)$ be the post-Lie algebra defined in Proposition~\ref{prop:homogen}, and let $B=U(A)=\Omega^0(\M,U(\g))$. For $f,g\in B$ the frozen composition is $(fg)(p) = f(p)g(p)$, where we `freeze' the value of $f$ and $g$ at a point $p\in\M$ and obtain the product from $U(\g)$. The composition $f,g\mapsto f\opr g$, on the other hand, corresponds to the fundamental operation of composing two differential operators on $\M$. For $f,g\in \Der(B)$ we have
$f\opr g = fg + f\tr g$, splitting up the composition into a term $fg$, where $g$ is `frozen' (constant), and a term $f\tr g$ where the variation of $g$ along $f$ is taken into account.

In the free D-algebra $D_\C$, the composition of two forests $\omega_1,\omega_2\in \OFC$ is computed as~(\cite{munthe-kaas2008oth} Definition 2):
\begin{equation}\label{eq:grossmanprod}
\omega_1\opr\omega_2 = B^-(\omega_1\tr B^+(\omega_2)).
\end{equation}
The color of the added root $B^+$ does not matter, since the root is subsequently removed by $B^-$.
We call this the planar Grossman--Larson product, since it is a planar forest analogue of the Grossman--Larson product~\cite{grossman89has} of unordered trees appearing in the Connes--Kreimer Hopf algebra.

\subsection{Hopf algebras}
Hopf algebraic structures related to the free D-algebra $\D_\C= U(\fFCT(\C))$ have been studied in~\cite{ lundervold2011bea,lundervold2009hao,munthe-kaas2008oth}. These Hopf algebras can both be seen as generalizations of the shuffle--concatenation Hopf algebras of free Lie algebras as well as of the Connes--Kreimer Hopf algebra, which is closely related to pre-Lie algebras \cite{chapoton2001pla}.

\paragraph{Shuffle product.} From the classical theory of free Lie algebras, it follows that the derivations $\Der(\D_\C)$ can be characterized in terms of shuffle products. Define the shuffle product $\sh: \D_{\C}\tpr\D_{\C} \rightarrow \D_{\C}$ on the free D-algebra $\D_{\C}$ by $\one \sh \omega = \omega = \omega \sh \one$ and $$(\tau_1\omega_1) \sh (\tau_2 \omega_2) = \tau_1(\omega_1 \sh \tau_2 \omega_2) + \tau_2(\tau_1\omega_1 \sh \omega_2)$$ for $\tau_1, \tau_2 \in \OT$, $\omega_1, \omega_2 \in \OF$.
Let $(\cdot,\cdot)$ be an inner product on $\D_\C$ defined such that the forests form an orthonormal basis, and let the coproduct $\cpds\colon \D_{\C}\rightarrow  \D_{\C}\tpr\D_{\C}$ be the adjoint of $\sh$.

\begin{proposition}\label{prop:hnstar} The free D-algebra $\D_\C$ has the structure of a cocommutative Hopf algebra $\Hn'=\{\k\{\OFC\},\epsilon,\opr,\eta,\cpds,S\}$, whose product is the planar Grossman--Larson product $\opr$ defined in~(\ref{eq:grossmanprod}), coproduct $\cpds$ is the adjoint of the shuffle, and unit $\eta$ and counit $\epsilon$ are given as 
\begin{eqnarray*}
\eta(1)&=& \one\\
\epsilon(\one)&=& 1,\quad
\epsilon(\omega) = 0\quad\mbox{ for all $\omega\in \OFC\backslash\{\one\}$.}
\end{eqnarray*}
 The primitive elements are
$\mbox{Prim}(\Hn') = \Der(\D_\C)$.
The antipode $S$ is defined in~\cite{munthe-kaas2008oth}.
\end{proposition}

\begin{proof} The Hopf algebraic structure (for the dual of $\Hn'$) is proven in~\cite{munthe-kaas2008oth}. The characterization of the primitive elements follows from the free Lie algebra structure~\cite{reutenauer93fla}.
\end{proof}

\paragraph{The Hopf algebra $\Hn$, a magmatic view}  In the study of numerical integration on manifolds, it is important to characterize flows and parallel transport on manifolds with connections algebraically. It is convenient to base this on the dual Hopf algebra of $\Hn'$. Let $\Hn=\{\k\{\OFC\},\epsilon,\sh,\eta,\cpg,S\}$ be the commutative Hopf algebra of planar forests, where the product is the shuffle product $\sh$ and the coproduct $\cpg$ the adjoint of the planar Grossman--Larson product. Various expressions for $\cpg$ and the antipode $S$ are derived in~\cite{munthe-kaas2008oth}.
Our definition of $\OFC$ and $\Hn$ is rather involved, going via trees and enveloping algebras, extending $\tr$ from derivations, introducing
the dipterous composition $\opr$ and dualizing to obtain $\cpg$. However, both $\OFC$ and the Hopf algebra $\Hn$ can alternatively be defined in a compact, recursive manner. We will review this definition, which is the foundation for a computer implementation of $\Hn$ currently under construction.

\begin{definition}[Magmatic definition of $\OFC$]
Given a set $\C$ we let $\{\times_c\}_{c\in\C}$ be  a collection of magmatic products, i.e.\ a collection of products without any defining relations. Letting $\one$ denote the unity, we define $\OFC$ as the free magma generated from $\one$ by the magmatic products.
\end{definition}

This definition can be related to our previous definition of $\OFC$ by interpreting $\omega_1\times_c \omega_2$ in terms of forests as
\begin{equation}\omega_1\times_c\omega_2 = \omega_1 B^+_c (\omega_2)\end{equation}
for all $\omega_1,\omega_2\in \OFC$, $c\in \C$.
E.g., for a white colored node $c=\AB$, we have $\one\times_c \one = \AB$ and
\[\ab\times_c \aabb\ab = \ab \AaabbabB.\]
Any $\omega\in \OFC\backslash\{\one\}$ can be written uniquely as $\omega = \omega_L\times_c\omega_R$, where $c\in \C$ is the root of the rightmost tree in the forest. We call $\omega_L$ and $\omega_R$ the left and right parts of $\omega$ and $c$ the right root.

\begin{definition}[Shuffle product] The shuffle product $\sh\colon\k\{\OFC\}\tpr \k\{\OFC\}\rightarrow\k\{\OFC\}$ is defined by $\k$-linearity and the recursion
\begin{equation}\label{eq:shuffle} 
\begin{split}
\one\sh \omega &= \omega\sh \one = \omega, \quad \mbox{for all $\omega\in \OFC$,}\\
v\sh\omega &= (v_L\sh\omega)\times_c v_R+ (v\sh\omega_L)\times_d\omega_R, \quad\mbox{ for $v = v_L\times_c v_R$, 
$\omega = \omega_L\times_d \omega_R$.}
\end{split}
\end{equation}
\end{definition}

\begin{definition}[Coproduct] The coproduct $\cpg\colon\k\{\OFC\}\rightarrow\k\{\OFC\}\tpr \k\{\OFC\}$  is defined by $\k$-linearity and the recursion
\begin{equation}\label{eq:coprod} 
\begin{split}
\cpg(\one) &= \one\tpr\one \\
\cpg(\omega) &= \omega\tpr \one + \cpg(\omega_L){\sh\!\times_d}\cpg(\omega_R),\quad\mbox{for $\omega= \omega_L\times_d \omega_R$,}
\end{split}
\end{equation}
where $\sh\!\times_d$ is the shuffle product on the left and the magmatic product $\times_d$ on the right: \[(u_1\tpr u_2){\sh\!\times_d} (v_1\tpr v_2) := (u_1\sh v_1)\tpr(u_2\times_d v_2).\]
\end{definition}

\begin{proposition}[\cite{munthe-kaas2008oth}]$\Hn=\{\k\{\OFC\},\epsilon,\sh,\eta,\cpg,S\}$ is a commutative Hopf algebra.
\end{proposition}
The Hopf algebra $\Hn$ is the natural setting for \emph{Lie--Butcher series}.

\section{Lie--Butcher series and moving frames}\label{sec:LBseries}
Lie--Butcher series are formal power series for flows and vector fields on manifolds, which combine Butcher series with Lie series. An extensive overview can be found in~\cite{lundervold2009hao}. We will briefly review the basic definitions and present some new relations between LB-series and moving frames. In particular, we will see that one form of LB-series can be interpreted as a Taylor-type series for the development of a curve on a manifold. This interpretation is closely related to  moving frames, and can provide new geometric insight into the design of numerical integration algorithms.

\subsection{Lie--Butcher series, basic definitions}
\begin{definition}[Lie--Butcher series] Let $\Hn^* = \operatorname{Hom}_{\k}(\Hn,\k)$ denote the linear dual space of $\Hn$.  An element $\alpha\in \Hn^*$ is called a Lie--Butcher series. We identify $\alpha$ with an infinite series
\[\alpha= \sum_{\omega\in \OFC}\alpha(\omega) \omega,\]
via a dual pairing $(\cdot,\cdot)\colon \Hn^*\times \Hn\rightarrow \k$ defined such that 
\[\alpha(\omega) = (\alpha,\omega)\quad\mbox{for all $\omega\in \OFC$.}\]
\end{definition}

Define  \emph{characters} $G(\Hn)\subset\Hn^*$ and \emph{infinitesimal characters} $\g(\Hn) \subset\Hn^*$ as
 \begin{eqnarray}
 G(\Hn) & =&\big\{\alpha\in \Hn^* \colon \alpha(\one)=1,\, \,\alpha(\omega_1\sh\omega_2) = \alpha(\omega_1)\alpha(\omega_2) \mbox{ for $\omega_1,\omega_2\in \OFC$}\big\}\\
 \g(\Hn) & = & \big\{\alpha\in \Hn^* \colon \alpha(\one)=0,\, \,\alpha(\omega_1\sh\omega_2) = 0 \mbox{ for $\omega_1,\omega_2\in \OFC\backslash\{\one\}$}\big\} .
 \end{eqnarray}
The convolution product on $\Hn^*$ is defined in the standard way for Hopf algebras:
 \begin{equation}\alpha\opr \beta (\omega) = \sh (\alpha\tpr \beta)\cpg(\omega).\end{equation}
For characters taking values in $k$, the shuffle product reduces to the scalar product in $k$. The convolution is the extension of the planar Grossman--Larson product from finite series to infinite series, obtained by considering $\Hn^*$ as the projective limit $\Hn^* =\underleftarrow{\lim} \, \N_k$, where $\N_k = \mbox{span}\{\omega\in \OF \colon |\omega|\leq k\}$. Note that the series are \emph{formal} power series. The question of convergence for concrete realisations is outside the scope of the algebraic theory. 

$G(\Hn)$ with the convolution product forms a group called the character group of $\Hn$, where the unit and the inverse is given by the unit and the antipode in the Hopf algebra $\Hn$, see~\cite{lundervold2009hao}. In the special case where the post-Lie algebra is  pre-Lie, this is the Butcher group, first introduced in~\cite{butcher1972aat} as a tool to study numerical integration. More generally, the elements in $\g(\Hn)$ can represent vector fields, and elements in $G(\Hn)$  diffeomorphisms  on a manifold $\M$. The convolution represents the composition of diffeomorphisms. Parallel transport of $g\in \Hn^*$ along the  $t=1$ flow of $f\in \g(\Hn)$ is represented by the exponential 
of the connection, which using~(\ref{eq:dipt2}) becomes
\[\exp(f\tr)g:=g+f\tr g + \frac12f\tr(f\tr g)+\cdots = (\one+f+\frac12 f\opr f+\cdots)\tr g = \exp^\opr f \tr g,\]
where $\exp^\opr$ is the exponential with respect to the composition product. The map $\exp^\opr \colon \g(\Hn)\rightarrow G(\Hn)$ is 1--1, with inverse given by $\log^\opr$, see~\cite{lundervold2009hao} for 
explicit formulas for the logarithm.

\subsection{Moving frames and Lie--Butcher series for developments}\label{sec:lbhom} We will in this final section briefly discuss  $LB$-series in numerical analysis of integration on a homogeneous space $\M$. 
By formulating the numerical flows in terms of the post-Lie algebra structure defined in Section~\ref{sec:leftist} and choosing gauge by a moving frame, we obtain a new geometric insight from the 
 LB-series. 

We recall the basic operations: Let $G\times_{\! H} \g\rightarrow \M$ be the vector bundle~(\ref{eq:atiyahbundle}) and $\Gamma(G\times_{\! H} \g)$ the space of sections. Let ${\rho_{\! M}}\colon G\times_{\! H}\g\rightarrow \M$ be the finite motion~(\ref{eq:rhomap}) and ${\rho_{\! M}}_*\colon G\times_{\! H}\g\rightarrow \TM$ its infinitesimal generator~(\ref{eq:rhostarmap}). As shown in Proposition~\ref{prop:atiyahpostlie}, the space of sections $\Gamma(G\times_{\! H} \g)$ forms a post-Lie algebroid $\{G\times_{\! H} \g, [\cdot,\cdot],\tr,{\rho_{\! M}}_*\}$, where $[\cdot,\cdot]$ and $\tr$ on $\Gamma(G\times_{\! H} \g)$ correspond
to $-[\cdot,\cdot]$ and $\vartr$ on $\Omega_H^0(G,\g)$. Let $\sigma\colon \M\rightarrow G$ be a Cartan gauge (left partial moving frame), with $\sigma_*\colon \TM\rightarrow G\times_{\! H}\g$  defined in~(\ref{eq:sigmastar}). Note that ${\rho_{\! M}}_*\opr \sigma_* = \mbox{Id}_{\TM}$. 

%
Suppose we want to integrate a differential equation on $\M$ \[y'(t) = F(y(t)), \quad y(0)=y_0\in \M,\] where $F\in \XM$. We represent the vector field
as $F = {\rho_{\! M}}_*\opr f$, where
\[f = \sigma_*\opr F\in \Gamma(G\times_{\! H} \g).\]
A numerical method with timestep $0<h\in \RR$ is a diffeomorphism $\Psi_{hf}\colon \M\rightarrow \M$, designed to approximate the exact flow $\Phi_{hf}$. Numerical methods are constructed by composing the basic operations  $\rho_M$, $[\cdot,\cdot]$ and parallel transport with $\tr$, which in many applications can all be computed fast by linear algebra on finite dimensional spaces. The fact that the torsion $[\cdot,\cdot]$ in a post-Lie algebroid is tensorial makes this bracket much easier to compute than  the Lie algebroid bracket $\llbracket\cdot,\cdot\rrbracket$, which requires Lie derivation. 

In this setting the simplest possible numerical method, corresponding to Eulers classical integration scheme, reads
\begin{equation}\Psi^{\mbox{euler}}_{hf} = \rho_M\opr hf.\end{equation}
Many other integrators are discussed in  the literature~\cite{iserles2000lgm}.
All numerical methods constructed by composing the basic operations above can be represented as LB-series, and the analysis of the LB-series of the
numerical solution is a fundamental tool used to answer many questions about accuracy and geometric properties of the numerical integrators. 

Consider  $\Hn^*$, where $\C = \{\ab\}$, and an identification $\ab\mapsto f\in\Gamma(G\times_{\! H} \g)$. This extends uniquely to a map
$\F_f\colon \Hn^*\rightarrow \Gamma(G\times_{\! H} U(\g))$ into the enveloping algebra. For a forest $\omega\in \OFC$, $\F_f(\omega)$ is called an \emph{elementary differential operator}. 
Note that
for $t\in \RR$ we have $\F_{tf}(\omega) = t^{|\omega|}\F_f(\omega)$. B-series and LB-series are traditionally considered as time-dependent series of differential operators on $\RR^n$ and $\M$, respectively, given for $\alpha\in \Hn^*$ as
\[{\mathcal B}_{tf}(\alpha) := \sum_{\omega\in \OFC}t^{|\omega|}\alpha(\omega)\F_f(\omega) = \sum_{\omega\in \OFC}\alpha(\omega)\F_{tf}(\omega). \]

In~\cite[Section 4.3.4]{lundervold2009hao} we discuss three different ways the diffeomorphism $\Psi_{hf}$ may be represented in terms of a LB-series, and conversion between these representations. We will summarize these  and show that,  in the present setting, the third of them has an interesting geometric interpretation as an expansion of the development of the flow.

\paragraph{Pullback series (parallel transport).} Find $\alpha\in G(\Hn)$ such that 
\begin{equation}\label{eq:lie}{\mathcal B}_{hf}(\alpha)\tr g = \Psi_{hf}^* g, \quad \mbox{ for all $g\in \Gamma(G\times_{\! H} \g)$,} \end{equation}
where $\Psi_{hf}^*g$ denotes parallel transport of $g$ along $\Psi_{hf}$.

\paragraph{Backward error.} Find $\beta\in \g(\Hn)$ such that $\Psi_{hf}$ is exactly the $t=1$ flow of an autonomous vector field $\widetilde{F}_h = {\rho_{\! M}}_*\opr \widetilde{f}_h$, where
\begin{equation}\widetilde{f}_h = {\mathcal B}_{hf}(\beta)\in \Gamma(G\times_{\! H} \g).\end{equation}

\paragraph{Time dependent equation of Lie type (development).} Find $\gamma\in \g(\Hn)$ such that $y(h)=\Psi_{hf}(y_0)$ is the $t=h$ solution of the following
 time dependent equation of \emph{Lie type}
 \begin{equation}\label{eq:lietype}y'(t) = {\rho_{\! M}}_*\opr\left(\frac{\partial}{\partial t}{\mathcal B}_{tf}(\gamma)(y_0)\right), \qquad y(0)=y_0.\end{equation}

The curve $t\mapsto\frac{\partial}{\partial t}{\mathcal B}_{tf}(\gamma)(y_0)\in G\times_{\! H} \g$ is closely related to the \emph{development} of $y(t)$. Given a curve $t\mapsto y(t)\colon I\subset \RR\mapsto \M$ and a left  moving frame $\sigma\colon \M\rightarrow G$ with Darboux derivative $\omega_\sigma\in \Omega^1(\M,\g)$, the development of $y(t)$ is the curve $\delta\colon I\subset\RR\rightarrow \g$ defined as
$\delta(t):=\omega_\sigma \opr y'(t)$.
\begin{proposition} Let $y(t)\in \M$ be the solution of a differential equation
\[y'(t) = \widetilde{F}(y(t)),\qquad y(0)=y_0.\]
Let $\widetilde{f} = \sigma_*\opr\widetilde{F}\in \Gamma(G\times_{\! H} \g)$ be identified with $\overline{f}\in \Omega^0_H(G,\g)$ as in Proposition~\ref{prop:11cor}. 
Then the development of $y(t)$ with respect to $\sigma$ is given as
\begin{equation}\label{eq:delta}\delta(t) = \overline{f}(\sigma(y(t))).\end{equation}
\end{proposition}

\begin{proof} We have $\widetilde{f}(y(t)) = \sigma(y(t))\times_{\! H}\omega_\sigma\opr \widetilde{F}(y(t))$. Thus 
$\overline{f}(\sigma(y(t)) = \omega_\sigma\opr \widetilde{F}(y(t)) = \omega_\sigma \opr y'(t) =: \delta(t)$.
\end{proof}

The solution of~(\ref{eq:delta}) in terms of LB-series is exactly the third characterisation of a flow in~\cite[Section 4.3.4]{lundervold2009hao}. 
Under the identification $\Gamma(G\times_{\! H} \g)\simeq\Omega^0_H(G,\g)$, we find the development as
\[\delta(t) \simeq \frac{\partial}{\partial t}{\mathcal B}_{tf}(\gamma)(\sigma(y_0)).\]
We
end this section with an illustrative example.

\begin{example}The algebraic relationship between $\alpha$, $\beta$ and $\gamma$ is discussed in~\cite{lundervold2009hao}, where explicit formulas relating these three representations are presented:
We find $\alpha$ from $\beta$ by applying $\exp^\opr$ and $\beta$ from $\alpha$ via the eulerian idempotent in $\Hn$. From $\alpha$ we find $\gamma$ by applying the Dynkin idempotent in $\Hn$, and conversely 
$\alpha$ is found from $\gamma$ by a formula involving certain non-commutative Bell polynomials. 
 Consider the  \emph{exact}  flow $\Phi_{tf}$ of the differential equation $y'=F(y)$. In this case, obviously, $\beta = \ab$. From this we can compute explicitly the development $\delta(t)\in\g$ of the solution curve $y(t)\in \M$, see~\cite{lundervold2009hao},
\begin{align*}{\small
\delta(t) & = {\mathcal F}_f\Biggl(
\ab + t\ \aabb + \frac{t^2}{2!}\Biggl(\aababb+\aaabbb\Biggr) + \frac{t^3}{3!}\Biggl(\aabababb+\aaabbabb+2\aabaabbb
+\aaababbb+\aaaabbbb\Biggr)+ \frac{t^4}{4!}\Biggl(\aababababb \\
&  +\aaabbababb+2\aabaabbabb +3\aababaabbb+\aaababbabb+ \aaaabbbabb+3\aaabbaabbb+3\aabaababbb+3\aabaaabbbb+\aaabababbb+\aaaabbabbb\\
& +2\aaabaabbbb+\aaaababbbb+\aaaaabbbbb\Biggr)+\frac{t^5}{5!}\Biggl(\aabababababb+\cdots\Biggr)+\cdots \Biggr)(\sigma(y_0)),
}\end{align*}
where the LB-series is computed in $\Omega^0_H(G,\g)$.

This expression yields the Taylor expansions of invariants of the curve, obtained from the moving frame. 
\end{example}

\section*{Concluding remarks} 
The main goal of this essay has been to expose the importance of post-Lie algebras in the algebraic analysis of flows on Lie groups and homogeneous spaces. We have emphasized both geometric and algebraic aspects of the theory, and shown that Lie--Butcher series are closely related to moving frames. The paper opens up several interesting areas for further investigation. We are convinced that  post-Lie algebraic structures will find applications in many areas also outside numerical analysis, such as stochastic differential equations, control theory and sub-Riemannian geometry. The paper points out different ways of applying moving frame techniques in the choice of isotropy in Lie group integration. Many aspects of this still have to be investigated analytically and computationally.
 
\section*{Acknowledgements} We would like to thank Kurusch Ebrahimi-Fard, Dominique Manchon and  Olivier Verdier  for valuable discussions about the topics of this paper, and  thanks to Jon Eivind Vatne for explaining operads and for pointing us to the work of Bruno Vallette. Finally, many thanks to Matthias Kawski and the anonymous referees for their careful reading, and for their many suggested improvements of the original manuscript.


\bibliography{torsion}

\end{document}